\documentclass[oneside,english]{amsart}
\usepackage[T1]{fontenc}
\usepackage[latin9]{inputenc}
\usepackage{amsthm}
\usepackage{amstext}
\usepackage{amssymb}
\usepackage{esint}

\makeatletter
\numberwithin{equation}{section}
\numberwithin{figure}{section}
\theoremstyle{plain}
\newtheorem{thm}{\protect\theoremname}
  \theoremstyle{definition}
  \newtheorem{defn}[thm]{\protect\definitionname}
  \theoremstyle{plain}
  \newtheorem{lem}[thm]{\protect\lemmaname}
  \theoremstyle{plain}
  \newtheorem{prop}[thm]{\protect\propositionname}
  \theoremstyle{remark}
  \newtheorem{rem}[thm]{\protect\remarkname}

\makeatother

\usepackage{babel}
  \providecommand{\definitionname}{Definition}
  \providecommand{\lemmaname}{Lemma}
  \providecommand{\propositionname}{Proposition}
  \providecommand{\remarkname}{Remark}
\providecommand{\theoremname}{Theorem}

\begin{document}

\title{Expected signature of Gaussian processes with strictly regular kernels }

\author{Horatio Boedihardjo}

\author{Anastasia Papavasiliou}

\author{Zhongmin Qian}
\begin{abstract}
We compute the expected signature of a class of Gaussian processes
which is a subclass of the Gaussian processes with regular kernels,
in the sense of \cite{AMN01}. 
\end{abstract}
\maketitle

\section{Introduction }

The expected signature of a stochastic process was first studied by
T. Fawcett in \cite{Faw03} and N. Victoire in \cite{LV04}, who independently
calculated the expected signature of Brownian motion. This is used
in \cite{LV04} to calculate the cubature measure which approximates
the Wiener measure. Since then, interests in the subject have grown.
In \cite{LH11}, Ni and Lyons expressed the expected signature of
Brownian motion in a disc up to the first exit time in terms of the
solution of a PDE. In \cite{Wer12}, the first three gradings of the
expected signature of the Chordal SLE measure were explicitly calculated.
In this article we calculate the expected signature of Gaussian processes
with strictly regular kernel (defined in next section). These are
Gaussian processes with regular kernels (see \cite{AMN01}) which
do not have Brownian components. 

The original motivation for undertaking this study was to find the
expected signature of fractional Brownian motions for $H>\frac{1}{2}$.
After posting this result, we were informed that the calculation of
expected signature for fractional Brownian motions for Hurst parameter
$H>\frac{1}{3}$ has already appeared in \cite{Fabric}. Therefore,
this article is a generalisation of the calculation in \cite{Fabric}.

\section{Main result }

We will first recall some notations.

Let $T>0$ be fixed throughout this note. 

Let $\triangle^{\prime}:=\left\{ \left(t,s\right)\in\mathbb{R}^{2}:0\leq s<t\leq T\right\} $.

Let $K\left(\cdot,\cdot\right):\triangle^{\prime}\rightarrow\mathbb{R}$
be a function such that:

(K1) For each $r$, $K\left(\cdot,r\right)$ is absolutely continuous
on $(r,T]$. 

(K2) $K\left(r^{+},r\right)=0\;\forall r\in\left[0,T\right]$.

(K3) Let $\left|K\right|\left((r,T],r\right)$ denote the total variation
of $K\left(\cdot,r\right)$ on $(r,T]$. Then
\[
\int_{0}^{\infty}\left|K\right|\left((r,T],r\right)^{2}\mathrm{d}r<\infty.
\]

If we remove the condition (K2) and weaken the absolute continuity
condition in (K1) to bounded total variation, then we recover the
notion of regular kernel in \cite{AMN01}. These extra conditions
mean that our Gaussian processes have to be strictly smoother than
Brownian motion. 

Let $\partial_{1}K$ denote the derivative of $K$ with respect to
its first coordinate. 

Let $W_{t}$ be a Gaussian process of the form 
\begin{equation}
W_{t}:=\int_{0}^{t}K\left(t,r\right)\mathrm{d}B_{r},\quad t\in\left[0,T\right],\label{eq:defn}
\end{equation}
where $\mathrm{d}B_{r}$ denotes the integration in the Itô's sense.

We shall call a Gaussian process $W$ of the form (\ref{eq:defn}),
where $K$ satisfies (K1), (K2) and (K3), a Gaussian process with
a strictly regular kernel $K$. 

A $d$-dimensional Gaussian process with a strictly regular kernel
is a process whose coordinate components are independent and identically
distributed copies of a Gaussian process with a strictly regular kernel. 

The expected signature of such process will be expressed in terms
of pairings, which we shall recall below: 
\begin{defn}
A set $\pi$ is a pairing of $\left\{ 1,2,\ldots,2n-1,2n\right\} $
if there exists two injective functions $a:\left\{ 1,\ldots,n\right\} \rightarrow\left\{ 1,\ldots,2n\right\} $
and $b:\left\{ 1,\dots,n\right\} \rightarrow\left\{ 1,\ldots,2n\right\} $
such that $a\left(i\right)<b\left(i\right)$ for all $i$ and 
\[
\pi=\left\{ \left(a\left(i\right),b\left(i\right)\right):i=1,..,n\right\} 
\]

\end{defn}
Let $E_{2n}$ be a subset of $\left\{ 1,2,\ldots,d\right\} ^{2n}$
such that $\left(i_{1},\ldots,i_{2n}\right)\in E_{2n}$ if and only
if for all $k\in\left\{ 1,2,\ldots,d\right\} $, the set 
\[
\left\{ j:i_{j}=k\right\} 
\]
has an even number of elements. 

Let $\Pi_{2n}$ denote the set of all possible pairings of $\left\{ 1,2,\ldots,2n\right\} $.
Given $\left(i_{1},\ldots,i_{2n}\right)\in\left\{ 1,2,\ldots,d\right\} ^{2n}$,
define $\Pi_{i_{1},\ldots,i_{2n}}$ to be a subset of $\Pi_{2n}$
whose elements $\pi$ satisfy 
\[
\left(l,j\right)\in\pi\implies i_{l}=i_{j}
\]
We can now state our main result: 
\begin{thm}
\label{thm:gen} Let $W$ be a $d$-dimensional Gaussian process with
a strictly regular kernel $K$. If $k=2n$ for some $n\in\mathbb{N}$,
and $\left(i_{1},\ldots,i_{2n}\right)\in E_{2n}$, then the projection
to the basis $e_{i_{1}}\otimes e_{i_{2}}...\otimes e_{i_{k}}$ of
the expected signature of $W$ up to time $T$ is 
\[
\sum_{\pi\in\Pi_{i_{1},\ldots,i_{2n}}}\int_{\triangle_{2n}\left(T\right)}\Pi_{\left(l,j\right)\in\pi}\left[\int_{0}^{u_{j}\wedge u_{l}}\left[\partial_{1}K\left(u_{j},r\right)\right]\left[\partial_{1}K\left(u_{l},r\right)\right]\mathrm{d}r\right]\mathrm{d}u_{1}..\mathrm{d}u_{n}
\]
where $\triangle_{2n}\left(T\right)$ denotes the simplex $\left\{ \left(u_{1},\ldots,u_{2n}\right)\in\mathbb{R}^{2n}:0\leq u_{1}<\ldots<u_{2n}\leq T\right\} $.
The projection to the basis $e_{i_{1}}\otimes e_{i_{2}}\ldots\otimes e_{i_{k}}$
is zero otherwise.
\end{thm}
The key idea of the calculation, inspired by \cite{CQ00}, is to first
calculate the expected signature of the piecewise linear approximation
of the Gaussian process. This reduces the problem into computing joint
moments of Gaussian random variables, which can be done using Wick's
formula. 

In section 3, we shall recall the properties of Gaussian rough paths. 

In section 4, we shall prove the main result Theorem \ref{thm:gen}.

In section 5, we show that our formula coincide with the formula for
the fractional Brownian motion case in \cite{Fabric}.

In section 6, we show that the expected signature is right continuous
at $H=\frac{1}{2}$.

\thanks{We would like to thank Xi Geng for useful discussions.}

\section{Signature of Gaussian process }

\subsection{Gaussan process }

Recall that $\triangle^{\prime}:=\left\{ \left(t,s\right)\in\mathbb{R}^{2}:0\leq s<t\leq T\right\} $. 

Let $K\left(\cdot,\cdot\right):\triangle^{\prime}\rightarrow\mathbb{R}$
be a function satisfying the conditions (K1), (K2) and (K3) in Section
1.1, and $W:\left[0,T\right]\rightarrow\mathbb{R}$ be a Gaussian
process satisfying (\ref{eq:defn}). 

Then $W_{t}$ is a Gaussian process with the covariance function 
\begin{equation}
R\left(t,s\right):=\int_{0}^{t\wedge s}K\left(t,r\right)K\left(s,r\right)\mathrm{d}r\label{eq:kernel1}
\end{equation}

The integral in (\ref{eq:kernel1}) exists by condition (K3).

By (K1) and (K2), we have $\partial_{1}K\left(\cdot,r\right)\in L^{1}(r,T]$
and 
\[
K\left(t,u\right)=\int_{r}^{t}\partial_{1}K\left(u,r\right)\mathrm{d}u.
\]

Thus (\ref{eq:kernel1}) becomes 
\[
R\left(t,s\right)=\int_{0}^{t\wedge s}\int_{r}^{s}\int_{r}^{t}\partial_{1}K\left(u,r\right)\partial_{1}K\left(v,r\right)\mathrm{d}u\mathrm{d}v\mathrm{d}r.
\]

By Tonelli's theorem, we have the following equality: 
\begin{equation}
R\left(t,s\right)=\int_{0}^{s}\int_{0}^{t}\left[\int_{0}^{u\wedge v}\left[\partial_{1}K\left(u,r\right)\right]\left[\partial_{1}K\left(v,r\right)\right]\mathrm{d}r\right]\mathrm{d}u\mathrm{d}v.\label{eq:kernel2}
\end{equation}

We shall denote the function $\int_{0}^{u\wedge v}\left[\partial_{1}K\left(u,r\right)\right]\left[\partial_{1}K\left(v,r\right)\right]\mathrm{d}r$
by $f\left(u,v\right)$. 

Note in particular that $f\left(\cdot,\cdot\right)\in L^{1}\left(\left[0,T\right]\times\left[0,T\right]\right)$
by Tonelli's theorem and (K3). 

To summarise, we have 
\[
R\left(t,s\right)=\int_{0}^{t}\int_{0}^{s}f\left(u,v\right)\mathrm{d}u\mathrm{d}v
\]
where $f\in L^{1}\left(\left[0,T\right]\times\left[0,T\right]\right)$. 

This means that for $\sigma\leq\tau$ and $s\leq t$ in $\left[0,T\right]$,
we have 
\[
\mathbb{E}\left[\left(W_{t}-W_{s}\right)\left(W_{\tau}-W_{\sigma}\right)\right]=\int_{\sigma}^{\tau}\int_{s}^{t}f\left(u,v\right)\mathrm{d}u\mathrm{d}v.
\]

This expression will be key to our computation. 

By the definition of $f$, we also have $f\left(u,v\right)=f\left(v,u\right)$. 

We summarise our calculations in the following lemma:
\begin{lem}
\label{lem:fund}Let $W:\left[0,T\right]\rightarrow\mathbb{R}$ be
a Gaussian process with a strictly regular kernel $K$. Then there
exists an integrable function $f:\left[0,T\right]^{2}\rightarrow\mathbb{R}$
such that $f\left(u,v\right)=f\left(v,u\right)$, and 
\[
\mathbb{E}\left[\left(W_{t}-W_{s}\right)\left(W_{\tau}-W_{\sigma}\right)\right]=\int_{\sigma}^{\tau}\int_{s}^{t}f\left(u,v\right)\mathrm{d}u\mathrm{d}v.
\]

\end{lem}

\subsection{Geometric rough paths}

Let $T^{n}\left(\mathbb{R}^{d}\right)$ and $T\left(\mathbb{R}^{d}\right)$
denote the graded algebras on $\mathbb{R}^{d}$ defined by
\[
T^{n}\left(\mathbb{R}^{d}\right):=\oplus_{k=0}^{n}\left(\mathbb{R}^{d}\right)^{\otimes k}
\]
and 
\[
T\left(\mathbb{R}^{d}\right):=\oplus_{k=0}^{\infty}\left(\mathbb{R}^{d}\right)^{\otimes k}
\]
where $\left(\mathbb{R}^{d}\right)^{\otimes0}:=\mathbb{R}$. 

We shall define three projection maps as follow: 
\begin{enumerate}
\item $\pi_{n}$ will denote the projection map from $T\left(\mathbb{R}^{d}\right)$
to $\left(\mathbb{R}^{d}\right)^{\otimes n}$.
\item If $I=\left(i_{1},..,i_{n}\right)\in\mathbb{N}^{n}$, then $\pi_{I}$
denote the projection map onto the basis $e_{i_{1}}\otimes\cdots\otimes e_{i_{n}}$
.
\item $\pi^{\left(n\right)}$ will denote the projection of an element of
$T\left(\mathbb{R}^{d}\right)$ onto $T^{n}\left(\mathbb{R}^{d}\right)$.
\end{enumerate}
We equip $\left(\mathbb{R}^{d}\right)^{\otimes k}$ with a metric
by identifying $\left(\mathbb{R}^{d}\right)^{\otimes k}$ with $\mathbb{R}^{d^{k}}$,
and we equip $T^{n}\left(\mathbb{R}^{d}\right)$ with the metric
\[
\left|\mathbf{w}\right|_{T^{n}\left(\mathbb{R}^{d}\right)}:=\max_{1\leq k\leq n}\left|\pi_{k}\left(\mathbf{w}\right)\right|
\]

Let $p\geq1$ and let $\mathcal{V}^{p}\left(\mathbb{R}^{d}\right)$
denote the set of all continuous functions $f:\left[0,T\right]\rightarrow\mathbb{R}^{d}$
with finite $p$-variation, i.e. 
\begin{equation}
\left\Vert f\right\Vert _{p}^{p}:=\sup_{\mathcal{P}}\sum_{k}\left|f\left(t_{k+1}\right)-f\left(t_{k}\right)\right|^{p}<\infty\label{eq:variation}
\end{equation}
where the supremum is taken over all finite partitions $\mathcal{P}:=\left(t_{0},t_{1},..,t_{n-1},t_{n}\right)$,
with $0=t_{0}<t_{1}<\cdots<t_{n-1}<t_{n}=T$.

We will now define the lift of functions in $\mathcal{V}^{1}$ .
\begin{defn}
Let $\gamma\in\mathcal{V}^{1}\left(\mathbb{R}^{d}\right)$and let
$\triangle_{n}\left(s,t\right):=\left\{ \left(t_{1},\ldots,t_{n}\right):s<t_{1}<\cdots<t_{n}<t\right\} $.
The \textit{lift} of $\gamma$ is a function $S\left(\gamma\right):\left\{ \left(s,t\right):0\leq s\leq t\right\} \rightarrow T\left(\mathbb{R}^{d}\right)$
defined by 
\begin{equation}
S\left(\gamma\right)_{s,t}=1+\sum_{n=1}^{\infty}\int_{\triangle_{n}\left(s,t\right)}\mathrm{d}\gamma_{t_{1}}\otimes\ldots\otimes\mathrm{d}\gamma_{t_{n}}\label{eq:signature definition}
\end{equation}
where the sum $+$ is the direct sum operation in $T\left(\mathbb{R}^{d}\right)$
and the integrals are taken in the Lebesgue-Stieltjes sense. 

The signature of a path $\gamma\in\mathcal{V}^{1}\left(\mathbb{R}^{d}\right)$
is defined to be $S\left(\gamma\right)_{0,1}$.
\end{defn}
Note in particular that $\pi_{n}\left(S\left(\gamma\right)_{s,t}\right)=\int_{\triangle_{n}\left(s,t\right)}\mathrm{d}\gamma_{t_{1}}\otimes\ldots\otimes\mathrm{d}\gamma_{t_{n}}$
and will be called the $n$-th grading of the lift of $\gamma$. We
will denote $\pi^{\left(n\right)}\left(S\left(\gamma\right)_{s,t}\right)$
by $S_{n}\left(\gamma\right)_{s,t}$.

The signature of a path satisfies the Chen's identity: 
\[
S\left(\gamma\right)_{s,u}\otimes S\left(\gamma\right)_{u,t}=S\left(\gamma\right)_{s,t}\;\forall\;0\leq s\leq u\leq t\leq T
\]

For paths which do not have finite variation, such as sample paths
of Brownian motion, the signatures have to be defined using geometric
rough paths. They are constructed using the following metric. Let
$\triangle:=\left\{ \left(s,t\right):0\leq s\leq t\leq T\right\} $.
\begin{defn}
\label{p var}Let $n\in\mathbb{N}$ and $p\geq1$. Let $\mathcal{V}^{p}\left(T^{n}\left(\mathbb{R}^{d}\right)\right)$
denote the set of all continuous functions $w$ from $\triangle$
to $T^{n}\left(\mathbb{R}^{d}\right)$ such that

1.$\pi_{0}\left(w_{s,t}\right)\equiv1$.

2. $w$ satisfies 

\[
\max_{1\leq k\leq n}\sup_{D}\left(\sum_{l}\left|\pi_{k}\left(w_{t_{l-1},t_{l}}\right)\right|^{\frac{p}{k}}\right)^{\frac{k}{p}}<\infty
\]
where $\sup_{D}$ runs over all partitions $t_{0}=0<t_{1}<t_{2}<..<t_{n}=T$.

Let $w^{1},w^{2}$ be elements of $\mathcal{V}^{p}\left(T^{n}\left(\mathbb{R}^{d}\right)\right)$.
We define, for each $p\geq1$, a distance function between $w^{1}$
and $w^{2}$ by
\[
\rho_{p-var}^{\left(n\right)}\left(w^{1},w^{2}\right)=\max_{1\leq k\leq n}\sup_{D}\left(\sum_{l}\left|\pi_{k}\left(w_{t_{l-1},t_{l}}^{1}\right)-\pi_{k}\left(w_{t_{l-1},t_{l}}^{2}\right)\right|^{\frac{p}{k}}\right)^{\frac{k}{p}}
\]
where $\sup_{D}$ runs over all partitions $t_{0}=0<t_{1}<t_{2}<..<t_{n}=T$.

For $p\geq1$, let $\left\lfloor p\right\rfloor $ be the integer
part of $p$. The metric $\rho_{p-var}^{\left(\left\lfloor p\right\rfloor \right)}\left(\cdot,\cdot\right)$
defined on $\mathcal{V}^{p}\left(T^{\left\lfloor p\right\rfloor }\left(\mathbb{R}^{d}\right)\right)$
is known as the \textit{p-variation metric }and will be denoted by
$d_{p}\left(\cdot,\cdot\right)$. 
\end{defn}
Let $C^{p-var}\left(\left[0,T\right],T^{n}\left(\mathbb{R}^{d}\right)\right)$
denote the set of all continuous function $w$ from $\triangle$ to
$T^{n}\left(\mathbb{R}^{d}\right)$ such that $d_{p}\left(w,0\right)<\infty$. 

Let $L\left(T^{n}\left(\mathbb{R}^{d}\right)\right)$ denote the set
$\left\{ \left(s,t\right)\rightarrow\pi^{\left(n\right)}\left(S\left(w\right)_{s,t}\right):w\in\mathcal{V}^{1}\left(\mathbb{R}^{d}\right)\right\} $. 

Let $G\Omega_{p}\left(\mathbb{R}^{d}\right)$ be the completion of
the set $L\left(T^{\left\lfloor p\right\rfloor }\left(\mathbb{R}^{d}\right)\right)$
under the $p$-variation metric in $\mathcal{V}^{p}\left(T^{\left\lfloor p\right\rfloor }\left(\mathbb{R}^{d}\right)\right)$.
$G\Omega_{p}\left(\mathbb{R}^{d}\right)$ is called the space of \textit{$p-$geometric
rough paths. }

This means that a continuous function $w:\left[0,T\right]\rightarrow T^{\left\lfloor p\right\rfloor }\left(\mathbb{R}^{d}\right)$
lies in $G\Omega_{p}\left(\mathbb{R}^{d}\right)$ if we can approximate
$w$ in the $p$ -variation metric by signatures of paths which have
finite variations. One candidate of such approximation is the piecewise
linear approximation, defined as below: 
\begin{defn}
\label{interpol}Let $w:\left[0,T\right]\rightarrow\mathbb{R}^{d}$
be a continuous function. 

Let $D=\left(0=t_{0}<t_{1}<..<t_{n}=T\right)$ be a partition of $\left[0,T\right]$.
We define the piecewise linear interpolation of $w$ with respect
to the partition $\mathcal{D}$ by 
\[
w^{D}\left(t\right):=w\left(t_{i}\right)+\frac{w\left(t_{i+1}\right)-w\left(t_{i}\right)}{t_{i+1}-t_{i}}\left(t-t_{i}\right)\; t\in\left[t_{i},t_{i+1}\right].
\]

\end{defn}
One final fact we need about geometric rough paths is that:
\begin{thm}
\label{thm:extension}(\cite{Lyn98}) Let $p\geq1$. Let $w\in\mathcal{V}^{p}\left(T^{\left\lfloor p\right\rfloor }\left(\mathbb{R}^{d}\right)\right)$
and that $w$ is multiplicative, in the sense that 
\[
w_{s,u}\otimes w_{u,t}=w_{s,t}
\]
for all $0\leq s\leq u\leq t$. Then for all $N\geq\left\lfloor p\right\rfloor $,
there exists a unique multiplicative functional $S_{N}\left(w\right)\in\mathcal{V}^{p}\left(T^{n}\left(\mathbb{R}^{d}\right)\right)$
such that 
\[
\pi^{\left(\left\lfloor p\right\rfloor \right)}\left(S_{N}\left(w\right)\right)=w
\]

\end{thm}

\subsection{Signature of Gaussian processes }

Recall that $\triangle$ denotes $\left\{ \left(s,t\right):0\leq s\leq t\le T\right\} $.
Let $f:\triangle\times\triangle\rightarrow\mathbb{R}$ be a function.
We shall follow \cite{FV10} and use the notation $f\left(\begin{array}{c}
s,t\\
u,v
\end{array}\right)$ to denote, for $s\leq t$ and $u\leq v$, 
\[
f\left(\begin{array}{c}
s,t\\
u,v
\end{array}\right):=f\left(s,u\right)+f\left(t,v\right)-f\left(s,v\right)-f\left(t,u\right)
\]
and a function $f:\left[0,T\right]^{2}\rightarrow\mathbb{R}$ is said
to have finite $\rho$-variation if $\left|f\right|_{\rho-var;\left[0,T\right]^{2}}<\infty$,
where 
\begin{equation}
\left|f\right|_{p-var;\left[s,t\right]\times\left[u,v\right]}=\sup_{\begin{array}{c}
\left(t_{i}\right)\in\mathcal{D}\left(\left[s,t\right]\right)\\
\left(t_{i}^{\prime}\right)\in\mathcal{D}\left(\left[u,v\right]\right)
\end{array}}\left(\sum_{i,j}\left|f\left(\begin{array}{c}
t_{i},t_{i+1}\\
t_{j}^{\prime},t_{j+1}^{\prime}
\end{array}\right)\right|^{p}\right)^{\frac{1}{p}}\label{eq:pvar}
\end{equation}
and $\mathcal{D}\left(\left[a,b\right]\right)$ be the set of all
partitions of the interval $\left[a,b\right]$. 

For a function $\omega:\triangle\times\triangle\rightarrow[0,\infty)$,
we shall denote, for $\left[s,t\right]\times\left[u,v\right]\in\left[0,T\right]^{2}$,
\[
\omega\left(\left[s,t\right]\times\left[u,v\right]\right):=\omega\left(s,t,u,v\right)
\]

\begin{defn}
A 2D \textit{control} is a continuous function $\omega:\triangle\times\triangle\rightarrow[0,\infty)$
such that for all rectangles $R_{1}$,$R_{2}$,$R$ in $\left[0,T\right]\times\left[0,T\right]$,
such that if $R_{1}\cup R_{2}\subset R$ , $R_{1}\cap R_{2}=\emptyset$,
then 
\[
\omega\left(R_{1}\right)+\omega\left(R_{2}\right)\leq\omega\left(R\right)
\]
and that for all rectangles $R$ with zero Lebesgue measure in $\mathbb{R}^{2}$
, 
\[
\omega\left(R\right)=0
\]

\end{defn}

\begin{defn}
Let $\omega:\triangle\times\triangle\rightarrow[0,\infty)$ be a 2D
control and let $f:\triangle\times\triangle\rightarrow\mathbb{R}$
be a continuous function. We say that $\omega$ \textit{controls the
$p$-variation of $f$} if for all $\left[s,t\right]\times\left[u,v\right]\subset\left[0,T\right]^{2}$,
\[
\left|f\left(\begin{array}{c}
s,t\\
u,v
\end{array}\right)\right|^{p}\leq\omega\left(\left[s,t\right]\times\left[u,v\right]\right)
\]
\end{defn}
\begin{lem}
\label{lem:controlvar}(\cite{FV10}, Lemma 5.56) Let $f:\triangle\times\triangle\rightarrow\mathbb{R}$
be a continuous function. Then $f$ has finite $\rho$-variation if
and only if there exists a 2D control $\omega$ such that $\omega$
controls the $\rho$-variation of $f$. 
\end{lem}
We will now recall the existence and some properties of the signatures
of Gaussian processes with strictly regular kernels.

Let $X$ be a $d$-dimensional Gaussian process, then its covariance
function is defined as $R_{X}\left(s,t\right):=\mathbb{E}\left[X_{s}\otimes X_{t}\right]$. 

The following theorem gives the existence of the signatures of some
Gaussian processes and the approximation result that the limit of
$\mathbb{E}\left[S\left(X^{D}\right)_{0,T}\right]$ as $\left\Vert D\right\Vert \rightarrow0$
is $\mathbb{E}\left[S\left(X\right)_{0,T}\right]$: 
\begin{thm}
\label{thm:aproxlem}(\cite{FR}) Let $X=\left(X^{1},..,X^{d}\right):\left[0,T\right]\rightarrow\mathbb{R}^{d}$
be a centered, continuous Gaussian process on a probability space
$\left(\Omega,\mathcal{F},\mathbb{P}\right)$, with $X^{i},X^{j}$
being independent if $i\neq j$. Assume that the covariance function
of $X$, denoted by $R_{X}$, has finite $\rho$-variation for some
$\rho\in[1,2)$ and that there exists a finite constant $K$ such
that $\left|R_{X}\right|_{p-var,\left[0,T\right]^{2}}\leq K$ (see
(\ref{eq:pvar})). Then there exists a process $\mathbf{X}$ with
sample paths almost surely in $\mathcal{V}^{p}\left(T^{\left\lfloor p\right\rfloor }\left(\mathbb{R}^{d}\right)\right)$
for all $p\in\left(2\rho,4\right)$, such that for any $\gamma>\rho$
, $\frac{1}{\gamma}+\frac{1}{\rho}>1$ and any $q>2\gamma$ and $N\in\mathbb{N}$
there exists a constant $C=C\left(q,\rho,\gamma,K,N,T\right)$ such
that for all $r\geq1$, 
\[
\left|\rho_{q-var}^{\left(N\right)}\left(S_{N}\left(X^{\left(k\right)}\right),S_{N}\left(\mathbf{X}\right)\right)\right|_{L^{r}}\leq Cr^{\frac{N}{2}}\sup_{0\leq t\leq T}\left|X_{t}^{\left(k\right)}-X_{t}\right|_{L^{2}}^{1-\frac{\rho}{\gamma}}
\]
for any piecewise linear interpolation $X^{\left(k\right)}$of $X$,
where the width of the partition is smaller than $\frac{1}{k}$. \end{thm}
\begin{defn}
Let $X$ be a process satisfying the conditions in Theorem \ref{thm:aproxlem},
we shall denote the process $\mathbf{X}$ in Theorem \ref{thm:aproxlem}
by $S\left(X\right)$ and the expected signature of $X$ on $\left[0,T\right]$
is defined to be $\mathbb{E}\left[S\left(X\right)_{0,T}\right]$. \end{defn}
\begin{prop}
\label{prop:limit}Let $X$ be a $d$-dimensional Gaussian process
with a strictly regular kernel. Then the covariance function of $X$
has finite $1$-variation.

Moreover, let $D^{m}$ be the dyadic partition $D^{m}=\left(0,\frac{T}{2^{m}},\frac{2T}{2^{m}},..,\frac{\left(2^{m}-1\right)T}{2^{m}},T\right)$.
Then for all $N\in\mathbb{N}$, 
\[
\left|\mathbb{E}\left[S_{N}\left(X^{D^{m}}\right)\right]-\mathbb{E}\left[S_{N}\left(X\right)\right]\right|\rightarrow0
\]
as $m\rightarrow\infty$, where $\left|\cdot\right|$ is the norm
induced by identifying $\left(\mathbb{R}^{d}\right)^{\otimes n}$
with $\mathbb{R}^{d^{n}}$.\end{prop}
\begin{proof}
By Lemma \ref{lem:fund} and its notation, the $1$-variation of the
covariance function of $X$, $R_{X}$, is controlled by 
\[
\omega\left(\left[s,t\right]\times\left[u,v\right]\right):=\int_{s}^{t}\int_{v}^{u}\left|f\left(x,y\right)\right|\mathrm{d}x\mathrm{d}y
\]
Thus by Lemma \ref{lem:controlvar}, $R_{X}$ has finite $1$-variation.
Moreover, the total $1$-variation of $R_{X}$ on $\left[0,T\right]^{2}$
is bounded above by $\int_{0}^{T}\int_{0}^{T}\left|f\left(x,y\right)\right|\mathrm{d}x\mathrm{d}y$
. Therefore, Theorem \ref{thm:aproxlem} applies. 

Observe that we have, in the pathwise sense, the following inequality
for all $N\in\mathbb{N}$,$q>2$: 
\[
\left|S_{N}\left(X^{D^{m}}\right)_{0,T}-S_{N}\left(\mathbf{X}\right)_{0,T}\right|\leq\left|\rho_{q-var}^{\left(N\right)}\left(S_{N}\left(X^{D^{m}}\right),S_{N}\left(\mathbf{X}\right)\right)\right|
\]

Thus by Theorem \ref{thm:aproxlem}, the only thing to prove is 
\[
\sup_{0\leq t\leq T}\left|X_{t}^{D^{m}}-X_{t}\right|_{L^{2}}\rightarrow0
\]

Let $t_{k}^{m}:=\frac{k}{2^{m}}T$. Then for $t_{k}^{m}\leq t\leq t_{k+1}^{m}$,
we have 
\[
\left|X_{t}^{D^{m}}-X_{t}\right|\leq\left|X_{t_{k+1}^{m}}-X_{t}\right|\left(\frac{t-t_{k}^{m}}{t_{k+1}^{m}-t_{k}^{m}}\right)+\left|X_{t_{k}^{m}}-X_{t}\right|\left(\frac{t_{k+1}^{m}-t}{t_{k+1}^{m}-t_{k}^{m}}\right)
\]

Thus by Lemma \ref{lem:fund} and its notation, 
\begin{eqnarray*}
\left|X_{t}^{D^{m}}-X_{t}\right|_{L^{2}} & \leq & \left|X_{t_{k+1}^{m}}-X_{t}\right|_{L^{2}}\left(\frac{t-t_{k}^{m}}{t_{k+1}^{m}-t_{k}^{m}}\right)+\left|X_{t_{k}^{m}}-X_{t}\right|_{L^{2}}\left(\frac{t_{k+1}^{m}-t}{t_{k+1}^{m}-t_{k}^{m}}\right)\\
 & \leq & \left(\frac{t-t_{k}^{m}}{t_{k+1}^{m}-t_{k}^{m}}\right)\left(\int_{\left[t,t_{k+1}^{m}\right]^{2}}\left|f\left(u,v\right)\right|\mathrm{d}u\mathrm{d}v\right)^{\frac{1}{2}}\\
 &  & +\left(\frac{t_{k+1}^{m}-t}{t_{k+1}^{m}-t_{k}^{m}}\right)\left(\int_{\left[t_{k}^{m},t\right]^{2}}\left|f\left(u,v\right)\right|\mathrm{d}u\mathrm{d}v\right)^{\frac{1}{2}}\\
 & \leq & 2\left(\int_{\left[t_{k}^{m},t_{k+1}^{m}\right]^{2}}\left|f\left(u,v\right)\right|\mathrm{d}u\mathrm{d}v\right)^{\frac{1}{2}}
\end{eqnarray*}

Note 
\begin{eqnarray*}
\left(\sup_{0\leq t\leq1}\left|X_{t}^{D^{m}}-X_{t}\right|_{L^{2}}\right)^{2} & \leq & 4\max_{1\leq k\leq2^{m}}\int_{\left[t_{k}^{m},t_{k+1}^{m}\right]^{2}}\left|f\left(u,v\right)\right|\mathrm{d}u\mathrm{d}v\\
 & \leq & 4\int_{\cup_{k=1}^{2^{m}}\left[t_{k}^{m},t_{k+1}^{m}\right]^{2}}\left|f\left(u,v\right)\right|\mathrm{d}u\mathrm{d}v
\end{eqnarray*}

The set $\cup_{k=1}^{2^{m}}\left[t_{k}^{m},t_{k+1}^{m}\right]^{2}$
converge to a set with Lebesgue measure zero as $m\rightarrow\infty$,
and $f$ is integrable on $\mathbb{R}^{2}$. Thus 
\[
\left(\sup_{0\leq t\leq1}\left|X_{t}^{D^{m}}-X_{t}\right|_{L^{2}}\right)^{2}\rightarrow0
\]
as $m\rightarrow\infty$. 
\end{proof}

\section{Computation }

In this section we shall calculate the expected signatures of Gaussian
processes with regular kernels. 

A key tool in our calculation is the following Wick's formula (also
known as Isserlis' theorem):
\begin{thm}
\cite{Wick's formula}Let $\left(W_{1},\ldots,W_{n}\right)$ be a
Gaussian vector. Then 
\[
\mathbb{E}\left(\Pi_{i=1}^{n}W_{i}\right)=\begin{cases}
\begin{array}{c}
0\\
\sum_{\pi\in\Pi_{n}}\Pi_{\left(i,j\right)\in\pi}\mathbb{E}\left(W_{i}W_{j}\right)
\end{array} & \begin{array}{c}
\mbox{if}\; n\;\mbox{is odd}.\\
\mbox{if }n=2k\;\mbox{is even}.
\end{array}\end{cases}
\]

\end{thm}
We will now carry out some preliminary calculations. Recall that $f$
is the function defined in Lemma \ref{lem:fund}. 

Let $t_{k}^{m}$ denote the dyadic partition point $\frac{k}{2^{m}}T$.

Define $c_{i,j}$, with $1\leq i\leq2^{m},1\leq j\leq2^{m}$, by 
\[
c_{i,j}=\int_{\mathbb{R}^{2}}f\left(u,v\right)1_{\left[t_{i-1}^{m},t_{i}^{m}\right]\times\left[t_{j-1}^{m},t_{j}^{m}\right]}\mathrm{d}u\mathrm{d}v
\]

Recall that given a continuous path $W:\left[0,T\right]\rightarrow\mathbb{R}^{d}$,
we define the dyadic approximation of $W$ by the function 
\[
W\left(m\right)_{t}=\begin{cases}
\begin{array}{c}
W_{0}\\
\sum_{i}\left(W_{t_{k}^{m}}^{i}+\frac{2^{m}}{T}\left(W_{t_{k}^{m}}^{i}-W_{t_{k-1}^{m}}^{i}\right)\left(t-t_{k-1}^{m}\right)\right)e_{i},
\end{array} & \begin{array}{c}
t=0\\
t\in\left(t_{k-1}^{m},t_{k}^{m}\right]
\end{array}\end{cases}
\]

Let $W$ be a $d$-dimensional Gaussian process with a strictly regular
kernel, and let $W\left(m\right)$ be the random process obtained
from the dyadic linear interpolation of its sample paths. 

Let $\left(k_{1},\ldots,k_{2N}\right)$ be a finite sequence of natural
numbers, satisfying $1\leq k_{1}\leq k_{2}\leq..\leq k_{2N}\leq2^{m}$.
Then the symbol $\left|\#k=1\right|$ will denote the number of elements
in the set $\left\{ j:k_{j}=1\right\} $, and in general, let $\left|\#k=i\right|$
be the number of elements in the set $\left\{ j:k_{j}=i\right\} $.

The following lemma summarises our preliminary calculation: 
\begin{lem}
\label{lem:cal}If $k=2n$ for some $n$ and $i_{1},..,i_{2n}\in E_{2n}$,
then the projection onto the basis $e_{i_{1}}\otimes\cdots\otimes e_{i_{k}}$
of $\mathbb{E}\left(S\left(W\left(m\right)\right)_{0,T}\right)$ is
\[
\sum_{\pi\in\Pi_{i_{1},...,i_{2n}}}\sum_{1\leq k_{1}\leq k_{2}\leq\ldots\leq k_{2n}\leq2^{m}}\frac{1}{\left|\#k=1\right|!\ldots\left|\#k=2^{m}\right|!}\cdot\Pi_{\left(j,l\right)\in\pi}c_{k_{l},k_{j}}
\]
otherwise the projection onto $e_{i_{1}}\otimes\cdots\otimes e_{i_{k}}$
of $\mathbb{E}\left(S\left(W\left(m\right)\right)_{0,T}\right)$ is
zero. \end{lem}
\begin{rem}
In fact we can prove by induction that 
\[
\frac{1}{\left|\#k=1\right|!\ldots\left|\#k=2^{m}\right|!}=\int_{\left[t_{k_{1}-1}^{m},t_{k_{1}}^{m}\right]\times\ldots\times\left[t_{k_{2N}-1}^{m},t_{k_{2N}}^{m}\right]\cap\triangle_{2N}}1\mathrm{d}u_{1}\ldots\mathrm{d}u_{2N}
\]
but we shall not need it here.\end{rem}
\begin{proof}
Define $X_{k}:\left[t_{k-1}^{m},t_{k}^{m}\right]\rightarrow\mathbb{R}^{d}$
by
\[
X_{k}\left(t\right)=\sum_{i}\left(W_{t_{k}^{m}}^{i}+\frac{2^{m}}{T}\left(W_{t_{k}^{m}}^{i}-W_{t_{k-1}^{m}}^{i}\right)\left(t-t_{k-1}^{m}\right)\right)e_{i}
\]
Then 
\[
\begin{array}{ccc}
S\left(X_{k}\right) & = & 1+\sum_{n=1}^{\infty}\int_{t_{k-1}^{m}<s_{1}<..<s_{n}<t_{k}^{m}}dX_{s_{1}}\otimes\ldots\otimes dX_{s_{n}}\\
 & = & 1+\sum_{n=1}^{\infty}\int_{t_{k-1}^{m}<s_{1}<..<s_{n}<t_{k}^{m}}\left[\sum_{i}\frac{2^{m}}{T}\left(W_{t_{k}^{m}}^{i}-W_{t_{k-1}^{m}}^{i}\right)e_{i}\right]^{\otimes n}ds_{1}\ldots ds_{n}\\
 & = & 1+\sum_{n=1}^{\infty}\left[\sum_{i}\frac{2^{m}}{T}\left(W_{t_{k}^{m}}^{i}-W_{t_{k-1}^{m}}^{i}\right)e_{i}\right]^{\otimes n}\int_{t_{k-1}^{m}<s_{1}<..<s_{n}<t_{k}^{m}}ds_{1}\ldots ds_{n}\\
 & = & 1+\sum_{n=1}^{\infty}\left[\sum_{i}\frac{2^{m}}{T}\left(W_{t_{k}^{m}}^{i}-W_{t_{k-1}^{m}}^{i}\right)e_{i}\right]^{\otimes n}\frac{1}{n!}\left(t_{k}^{m}-t_{k-1}^{m}\right)^{n}\\
 & = & e^{\sum_{i}\left(W_{t_{k}^{m}}^{i}-W_{t_{k-1}^{m}}^{i}\right)e_{i}}
\end{array}
\]

By Chen's identity,
\begin{equation}
\begin{array}{ccc}
S\left(W\left(m\right)_{t}\right) & = & \mathbb{E}\left[e^{\sum_{i}\left(W_{t_{1}^{m}}^{i}-W_{t_{0}^{m}}^{i}\right)e_{i}}\otimes\ldots\otimes e^{\sum_{i}\left(W_{t_{2^{m}}^{m}}^{i}-W_{t_{2^{m}-1}^{m}}^{i}\right)e_{i}}\right]\end{array}\label{eq:m1}
\end{equation}

We have 
\[
e^{\sum_{i}\left(W_{k}^{i}-W_{k-1}^{i}\right)e_{i}}=\sum_{j=0}^{\infty}\frac{1}{j!}\left(\sum_{i}\left(W_{t_{k}^{m}}^{i}-W_{t_{k-1}^{m}}^{i}\right)e_{i}\right)^{\otimes j}
\]

The coefficient of $e_{i_{1}}\otimes e_{i_{2}}\otimes\ldots\otimes e_{i_{n}}$
in the expansion of 
\[
\otimes_{k=1}^{2^{m}}e^{\sum_{i}\left(W_{k}^{i}-W_{k-1}^{i}\right)e_{i}}=\otimes_{k=1}^{2^{m}}\left[\sum_{j=0}^{\infty}\frac{1}{j!}\left(\sum_{i}\left(W_{t_{k}^{m}}^{i}-W_{t_{k-1}^{m}}^{i}\right)e_{i}\right)^{j}\right]
\]
is
\[
\sum_{1\leq k_{1}\leq k_{2}\leq\ldots\leq k_{n}\leq2^{m}}\frac{\Pi_{j=1}^{n}\left(W_{t_{k_{j}}^{m}}^{i_{j}}-W_{t_{k_{j}-1}^{m}}^{i_{j}}\right)}{\left|\#k=1\right|!\ldots\left|\#k=2^{m}\right|!}
\]

If $n$ were odd, then as the expected value of the product of an
odd number of Gaussian random variables is zero, we have the coefficient
of $e_{i_{1}}\otimes\ldots\otimes e_{i_{n}}$ in $\mathbb{E}\left(S\left(W\left(m\right)\right)_{0,t}\right)$
being zero. 

If $\left(i_{1},..,i_{2n}\right)\notin E_{2n}$, then since the process
$\left(W_{t}:t\geq0\right)$ has independent components and that the
expected value of the product of an odd number of Gaussian random
variables is zero, thus 
\[
\mathbb{E}\left(\Pi_{j=1}^{n}\left(W_{t_{k_{j}}^{m}}^{i_{j}}-W_{t_{k_{j}-1}^{m}}^{i_{j}}\right)\right)=0,
\]
which in turn implies that

\[
\begin{array}{c}
\pi^{i_{1},..,i_{2n}}\mathbb{E}\left[S\left(W\left(m\right)_{0,1}\right)\right]=0\end{array},
\]
when $\left(i_{1},..,i_{2n}\right)\notin E_{2n}$. 

We now calculate the projection to $e_{i_{1}}\otimes\cdots\otimes e_{i_{2n}}$
of $\mathbb{E}\left(S\left(W\left(m\right)\right)_{0,t}\right)$,
which by Wick's formula equals 
\begin{equation}
\begin{array}{cc}
 & \pi_{i_{1},...,i_{2n}}\left[\mathbb{E}\left(S\left(W\left(m\right)\right)_{0,t}\right)\right]\\
= & \sum_{\pi\in\Pi_{i_{1},...,i_{2n}}}\sum_{1\leq k_{1}\leq\ldots\leq k_{2n}\leq2^{m}}\frac{1}{\left|\#k=1\right|!...\left|\#k=2^{m}\right|!}\Pi_{\left(l,j\right)\in\pi}\\
 & \mathbb{E}\left[\left(W_{t_{k_{j}}^{m}}^{i_{j}}-W_{t_{k_{j}-1}^{m}}^{i_{j}}\right)\left(W_{t_{k_{l}}^{m}}^{i_{j}}-W_{t_{k_{l}-1}^{m}}^{i_{j}}\right)\right]\\
= & \sum_{\pi\in\Pi_{i_{1},...,i_{2n}}}\sum_{1\leq k_{1}\leq\ldots\leq k_{2n}\leq2^{m}}\frac{1}{\left|\#k=1\right|!\ldots\left|\#k=2^{m}\right|!}\cdot\Pi_{\left(l,j\right)\in\pi}c_{k_{l},k_{j}}
\end{array}\label{eq:1}
\end{equation}
where $\sum_{\pi\in\Pi_{i_{1},...,i_{2n}}}$ is the sum over all possible
pairings $\left(l,j\right)$, $j>l$ from the set $\Pi_{i_{1},...,i_{2n}}$.
\end{proof}
The following lemma is crucial to our calculation of the sum in Lemma
\ref{lem:cal}.
\begin{lem}
\textup{\label{lem:3 dia}For $i\geq1$ and any pairing $\pi$ of
$\left\{ 1,..,2n\right\} $, 
\[
\sum_{1\leq k_{1}\leq..\leq k_{i}=k_{i+1}\leq k_{i+2}\ldots\leq k_{2n}\leq2^{m}}\Pi_{\left(l,j\right)\in\pi}\left|c_{k_{l},k_{j}}\right|\rightarrow0\;\mbox{as}\; m\rightarrow\infty
\]
}\end{lem}
\begin{proof}
If $\left(i,i+1\right)\in\pi$, then 
\begin{equation}
\begin{array}{cc}
 & \sum_{1\leq k_{1}\leq\ldots\leq k_{i}=k_{i+1}\leq k_{i+2}\ldots\leq k_{2n}\leq2^{m}}\Pi_{\left(j,l\right)\in\pi}\left|c_{k_{l},k_{j}}\right|\\
\leq & \sum_{k=1}^{2^{m}}\left|c_{k,k}\right|\times\Pi_{\left(l,j\right)\in\pi\backslash\left(i,i+1\right)}\sum_{k_{j}=1}^{2^{m}}\sum_{k_{l}=1}^{k_{j}}\left|c_{k_{l},k_{j}}\right|
\end{array}\label{eq:first}
\end{equation}

Note that for $\left(j,l\right)\neq\left(i,i+1\right)$, 
\begin{equation}
\begin{array}{cc}
 & \sum_{k_{j}=1}^{2^{m}}\sum_{k_{l}=1}^{k_{j}}\left|c_{k_{l}k_{j}}\right|\\
= & \sum_{k_{j}=1}^{2^{m}}\sum_{k_{l}=1}^{k_{j}}\int_{\mathbb{R}^{2}}\left|f\left(u,v\right)\right|1_{\left[t_{k_{l-1}}^{m},t_{k_{l}}^{m}\right]\times\left[t_{k_{j}-1}^{m},t_{k_{j}}^{m}\right]}\mathrm{d}u\mathrm{d}v\\
\leq & \int_{\left[0,T\right]\times\left[0,T\right]}\left|f\left(u,v\right)\right|\mathrm{d}u\mathrm{d}v
\end{array}\label{eq:2}
\end{equation}

and that 
\begin{eqnarray*}
\sum_{k=1}^{2^{m}}\left|c_{k,k}\right| & \leq & \sum_{k=1}^{2^{m}}\int_{\mathbb{R}^{2}}\left|f\left(u,v\right)\right|1_{\left[t_{k}^{m},t_{k-1}^{m}\right]\times\left[t_{k}^{m},t_{k-1}^{m}\right]}\mathrm{d}u\mathrm{d}v\\
 & = & \int_{\mathbb{R}^{2}}\left|f\left(u,v\right)\right|1_{\cup_{k=1}^{2^{m}}\left[t_{k}^{m},t_{k-1}^{m}\right]\times\left[t_{k}^{m},t_{k-1}^{m}\right]}\left(u,v\right)\mathrm{d}u\mathrm{d}v.
\end{eqnarray*}

Since $f\left(\cdot,\cdot\right)\in L^{1}\left(\left[0,T\right]^{2}\right)$
and the set $\cup_{k=1}^{2^{m}}\left[t_{k}^{m},t_{k-1}^{m}\right]\times\left[t_{k}^{m},t_{k-1}^{m}\right]$
converges to a null set in $\mathbb{R}^{2}$ as $m\rightarrow\infty$,
we have 
\[
\sum_{k=1}^{2^{m}}\left|c_{k,k}\right|\rightarrow0
\]
as $m\rightarrow\infty$. 

Thus by (\ref{eq:first}) we have 
\[
\begin{array}{cc}
 & \sum_{1\leq k_{1}\leq\ldots\leq k_{i}=k_{i+1}\leq k_{i+2}\dots\leq k_{2n}\leq2^{m}}\Pi_{\left(j,l\right)\in\pi}\left|c_{k_{l}k_{j}}\right|\\
\leq & \left[\int_{\left[0,T\right]\times\left[0,T\right]}\left|f\left(u,v\right)\right|\mathrm{d}u\mathrm{d}v\right]^{n-1}\int_{\cup_{k=1}^{2^{m}}\left[t_{k}^{m},t_{k-1}^{m}\right]\times\left[t_{k}^{m},t_{k-1}^{m}\right]}\left|f\left(u,v\right)\right|\mathrm{d}u\mathrm{d}v\\
 & \rightarrow0\;\;\mbox{as}\;\;\; m\rightarrow\infty.
\end{array}
\]

Now if $\left(i,i+1\right)\notin\pi$, then let $\pi\left(i\right),\pi\left(i+1\right)$
denote the unique integers satisfying $\left(i,\pi\left(i\right)\right),\left(i+1,\pi\left(i+1\right)\right)\in\pi$. 

Assume now that $\left(i,i+1\right)\notin\pi$, then we have 
\begin{equation}
\begin{array}{cc}
 & \sum_{1\leq k_{1}\leq\ldots\leq k_{i}=k_{i+1}\leq k_{i+2}\ldots\leq k_{2n}\leq2^{m}}\Pi_{\left(l,j\right)\in\pi}\left|c_{k_{l}k_{j}}\right|\\
\leq & \Pi_{\left(l,j\right)\in\pi\backslash\left\{ \left(i,\pi\left(i\right)\right),\left(i+1,\pi\left(i+1\right)\right)\right\} }\sum_{k_{j}=1}^{2^{m}}\sum_{k_{l}=1}^{k_{j}}\left|c_{k_{l}k_{j}}\right|\\
 & \times\sum_{k=1}^{2^{m}}\sum_{k_{\pi\left(i\right)}=1}^{2^{m}}\sum_{k_{\pi\left(i+1\right)}=1}^{2^{m}}\left|c_{k,k_{\pi\left(i\right)}}\right|\left|c_{k,k_{\pi\left(i+1\right)}}\right|.
\end{array}\label{eq:3}
\end{equation}

Let $F\left(v\right):=\int_{0}^{T}\left|f\left(u,v\right)\right|\mathrm{d}u$.
Then as $f\left(\cdot,\cdot\right)\in L^{1}\left[0,T\right]^{2}$,
we have $F\left(\cdot\right)\in L^{1}\left[0,T\right]$. Thus 
\[
\sum_{k_{\pi\left(i\right)}=1}^{2^{m}}\int_{\mathbb{R}^{2}}\left|f\left(u,v\right)\right|1_{\left[t_{k_{\pi\left(i\right)}}^{m},t_{k_{\pi\left(i\right)}-1}^{m}\right]\times\left[t_{k}^{m},t_{k-1}^{m}\right]}\mathrm{d}u\mathrm{d}v=\int_{0}^{T}F\left(v\right)1_{\left[t_{k}^{m},t_{k-1}^{m}\right]}\left(v\right)\mathrm{d}v.
\]

Note that since $F\geq0$, and that $F\left(\cdot\right)\in L^{1}$,
the integral 
\[
\int_{0}^{T}\int_{0}^{T}F\left(u\right)F\left(v\right)\mathrm{d}u\mathrm{d}v=\int_{0}^{T}F\left(u\right)\mathrm{d}u\int_{0}^{T}F\left(v\right)\mathrm{d}v
\]
exist and thus $\left(u,v\right)\rightarrow F\left(u\right)F\left(v\right)$
is integrable on $\left[0,T\right]^{2}$. 

Hence

\begin{eqnarray*}
 &  & \sum_{k=1}^{2^{m}}\sum_{k_{\pi\left(i\right)}=1}^{2^{m}}\sum_{k_{\pi\left(i+1\right)}=1}^{2^{m}}\left|c_{k,k_{\pi\left(i\right)}}\right|\left|c_{k,\pi\left(i+1\right)}\right|\\
 & \leq & \sum_{k=1}^{2^{m}}\left\{ \left[\sum_{k_{\pi\left(i\right)}=1}^{2^{m}}\int_{\mathbb{R}^{2}}\left|f\left(u,v\right)\right|1_{\left[t_{k_{\pi\left(i\right)}}^{m},t_{k_{\pi\left(i\right)}-1}^{m}\right]\times\left[t_{k}^{m},t_{k-1}^{m}\right]}\mathrm{d}u\mathrm{d}v\right]\right.\\
 &  & \left.\times\left[\sum_{k_{\pi\left(i+1\right)}=1}^{2^{m}}\int_{\mathbb{R}^{2}}\left|f\left(u,v\right)\right|1_{\left[t_{k_{\pi\left(i+1\right)}}^{m},t_{k_{\pi\left(i+1\right)}-1}^{m}\right]\times\left[t_{k}^{m},t_{k-1}^{m}\right]}\mathrm{d}u\mathrm{d}v\right]\right\} \\
 & \leq & \sum_{k=1}^{2^{m}}\int_{\mathbb{R}^{2}}F\left(v\right)F\left(u\right)1_{\left[t_{k}^{m},t_{k-1}^{m}\right]\times\left[t_{k}^{m},t_{k-1}^{m}\right]}\left(u,v\right)\mathrm{d}u\mathrm{d}v.
\end{eqnarray*}

As $\left(u,v\right)\rightarrow F\left(u\right)F\left(v\right)$ is
integrable on $L^{1}\left[0,T\right]^{2}$, and the set $\cup_{k=1}^{2^{m}}\left[t_{k}^{m},t_{k-1}^{m}\right]\times\left[t_{k}^{m},t_{k-1}^{m}\right]$
converges to a null set in $\mathbb{R}^{2}$, so that 
\[
\sum_{k=1}^{2^{m}}\int_{\mathbb{R}^{2}}F\left(v\right)F\left(u\right)1_{\left[t_{k}^{m},t_{k-1}^{m}\right]\times\left[t_{k}^{m},t_{k-1}^{m}\right]}\left(u,v\right)\mathrm{d}u\mathrm{d}v\rightarrow0,
\]

as $m\rightarrow\infty$. 

Together with (\ref{eq:2}) and (\ref{eq:3}) we have 
\begin{eqnarray*}
 &  & \sum_{1\leq k_{1}\leq\ldots\leq k_{i}=k_{i+1}\leq k_{i+2}\ldots\leq k_{2n}\leq2^{m}}\Pi_{\left(l,j\right)\in\pi}\left|c_{k_{l}k_{j}}\right|\\
 & \leq & \left(\sum_{k=1}^{2^{m}}\int_{\mathbb{R}^{2}}F\left(v\right)F\left(u\right)1_{\left[t_{k}^{m},t_{k-1}^{m}\right]\times\left[t_{k}^{m},t_{k-1}^{m}\right]}\left(u,v\right)\mathrm{d}u\mathrm{d}v\right)\\
 &  & \times\left(\int_{\left[0,T\right]\times\left[0,T\right]}\left|f\left(u,v\right)\right|\mathrm{d}u\mathrm{d}v\right)^{n-2}\\
 & \rightarrow & 0,
\end{eqnarray*}

as $m\rightarrow\infty$. 

This completes the proof of the lemma.
\end{proof}
We now prove our main result Theorem \ref{thm:gen}. 
\begin{proof}
(of Theorem \ref{thm:gen}) By Proposition \ref{prop:limit}, the
expected signature is given by the limit as $m\rightarrow\infty$
of the sum in Lemma \ref{lem:cal}. 

However, by Lemma \ref{lem:3 dia}, for any $i\geq1$, as $m\rightarrow\infty$,
\[
\sum_{1\leq k_{1}\leq\ldots\leq k_{i}=k_{i+1}\leq k_{i+2}\ldots\leq k_{2n}\leq2^{m}}\frac{1}{\left|\#k=1\right|!\ldots\left|\#k=2^{m}\right|!}\cdot\Pi_{\left(l,j\right)\in\pi}\left|c_{k_{l},k_{j}}\right|\rightarrow0
\]

Thus 
\[
\begin{array}{cc}
 & \lim_{m\rightarrow\infty}\sum_{1\leq k_{1}\leq\ldots\leq k_{2n}\leq2^{m}}\frac{1}{\left|\#k=1\right|!\ldots\left|\#k=2^{m}\right|!}\cdot\Pi_{\left(j,l\right)\in\pi}c_{k_{l}k_{j}}\\
= & \lim_{m\rightarrow\infty}\sum_{1<k_{1}<\ldots<k_{2n}\leq2^{m}}\frac{1}{\left|\#k=1\right|!\ldots\left|\#k=2^{m}\right|!}\Pi_{\left(j,l\right)\in\pi}c_{k_{l}k_{j}}\\
= & \lim_{m\rightarrow\infty}\sum_{1<k_{1}<\ldots<k_{2n}\leq2^{m}}\Pi_{\left(j,l\right)\in\pi}c_{k_{l}k_{j}}
\end{array}
\]
where the last equality uses the fact that for $k_{1}<k_{2}<\ldots<k_{2N}$,
we have $\left|\#k=1\right|!\ldots\left|\#k=2^{m}\right|!=1$. 

Note that 
\[
\begin{array}{cc}
 & \Pi_{\left(j,l\right)\in\pi}c_{k_{l}k_{j}}\\
= & \Pi_{\left(l,j\right)\in\pi}\int_{\mathbb{R}^{2}}f\left(u,v\right)1_{\left[t_{k_{l}}^{m},t_{k_{l}-1}^{m}\right]\times\left[t_{k_{j}}^{m},t_{k_{j}-1}^{m}\right]}\mathrm{d}u\mathrm{d}v\\
= & \int_{\mathbb{R}^{2}}\Pi_{\left(j,l\right)\in\pi}f\left(u_{j},u_{l}\right)1_{\left[t_{k_{1}-1}^{m},t_{k_{1}}^{m}\right]\times\ldots\times\left[t_{k_{2n}-1}^{m},t_{k_{2n}}^{m}\right]}\left(u_{1},\ldots,u_{2n}\right)du_{1}...du_{2n}.
\end{array}
\]

Thus 
\begin{eqnarray*}
 &  & \lim_{m\rightarrow\infty}\sum_{1<k_{1}<k_{2}<\ldots<k_{2n}\leq2^{m}}\Pi_{\left(j,l\right)\in\pi}c_{k_{l}k_{j}}\\
 & = & \lim_{m\rightarrow\infty}\sum_{1<k_{1}<...<k_{2n}\leq2^{m}}\int_{\left[t_{k_{1}-1}^{m},t_{k_{1}}^{m}\right]\times\ldots\times\left[t_{k_{2n}-1}^{m},t_{k_{2n}}^{m}\right]}\Pi_{\left(j,l\right)\in\pi}f\left(u_{j},u_{l}\right)du_{1}\ldots du_{2n}.
\end{eqnarray*}

Let $\triangle_{2n}\left(0,T\right)$ denote the simplex $\left\{ \left(u_{1},..,u_{2n}\right)\in\mathbb{R}^{2n}:0\leq u_{1}<\ldots<u_{2n}\leq T\right\} $.
Then, 
\[
\begin{array}{cc}
 & \int_{\left[t_{k_{1}-1}^{m},t_{k_{1}}^{m}\right]\times...\times\left[t_{k_{2n}-1}^{m},t_{k_{2n}}^{m}\right]\cap\triangle_{2N}\left(0,T\right)}\begin{array}{c}
\Pi_{\left(l,j\right)\in\pi}\left|f\left(u_{j},u_{l}\right)\right|du_{1}\ldots du_{2n}\end{array}\\
\leq & \int_{\left[t_{k_{1}-1}^{m},t_{k_{1}}^{m}\right]\times...\times\left[t_{k_{2n}-1}^{m},t_{k_{2n}}^{m}\right]}\begin{array}{c}
\Pi_{\left(l,j\right)\in\pi}\left|f\left(u_{j},u_{l}\right)\right|du_{1}\ldots du_{2n}\end{array}\\
\leq & \Pi_{\left(j,l\right)\in\pi}\left|c_{k_{l}k_{j}}\right|.
\end{array}
\]

However, by Lemma \ref{lem:3 dia}, 
\[
\sum_{1\leq k_{1}\leq..\leq k_{i}=k_{i+1}\leq k_{i+2}\ldots\leq k_{2n}\leq2^{m}}\Pi_{\left(j,l\right)\in\pi}\left|c_{k_{l}k_{j}}\right|\rightarrow0\;\mbox{as}\; m\rightarrow\infty.
\]

Thus for any $1\leq i$, as $m\rightarrow\infty$, 
\begin{align*}
 & \sum_{1\leq k_{1}\leq..\leq k_{i}=k_{i+1}\leq k_{i+2}\ldots\leq k_{2n}\leq2^{m}}\int_{\left[t_{k_{1}-1}^{m},t_{k_{1}}^{m}\right]\times\ldots\times\left[t_{k_{2n}-1}^{m},t_{k_{2n}}^{m}\right]\cap\triangle_{2N}\left(0,T\right)}\\
 & \begin{array}{c}
\Pi_{\left(l,j\right)\in\pi}\left|f\left(u_{j},u_{l}\right)\right|du_{1}\ldots du_{2n}\end{array}
\end{align*}
converges to zero. 

Hence, 
\begin{eqnarray*}
 &  & \lim_{m\rightarrow\infty}\sum_{1<k_{1}<...<k_{2n}\leq2^{m}}\int_{\left[t_{k_{1}}^{m},t_{k_{1}-1}^{m}\right]\times\ldots\times\left[t_{k_{2n}}^{m},t_{k_{2n}-1}^{m}\right]}\Pi_{\left(j,l\right)\in\pi}f\left(u_{j},u_{l}\right)du_{1}\ldots du_{2n}\\
 & = & \lim_{m\rightarrow\infty}\sum_{1\leq k_{1}\leq...\leq k_{2n}\leq2^{m}}\int_{\left[t_{k_{1}}^{m},t_{k_{1}-1}^{m}\right]\times\ldots\times\left[t_{k_{2n}}^{m},t_{k_{2n}-1}^{m}\right]\cap\triangle_{2N}\left(0,T\right)}\Pi_{\left(j,l\right)\in\pi}f\left(u_{j},u_{l}\right)du_{1}\ldots du_{2n}\\
 & = & \int_{\triangle_{2n}}\begin{array}{c}
\Pi_{\left(l,j\right)\in\pi}f\left(u_{j},u_{l}\right)du_{1}\ldots du_{2n}\end{array}.
\end{eqnarray*}

Finally by Lemma \ref{lem:cal} and Lemma \ref{prop:limit}, 
\begin{eqnarray*}
 &  & \pi_{i_{1},..,i_{2n}}\left(\mathbb{E}\left(S\left(W\right)\right)_{0,T}\right)\\
 & = & \lim_{m\rightarrow\infty}\pi_{i_{1},..,i_{2n}}\left(\mathbb{E}\left(S\left(W\left(m\right)\right)\right)_{0,T}\right)\\
 & = & \sum_{\pi\in\Pi_{i_{1},...,i_{2n}}}\int_{\triangle_{2N}}\begin{array}{c}
\Pi_{\left(l,j\right)\in\pi}f\left(u_{j},u_{l}\right)du_{1}\ldots du_{2n}\end{array}.
\end{eqnarray*}

\end{proof}

\section{The fractional Brownian motion with Hurst parameter $H>\frac{1}{2}$}

We will now show that the formula we give coincide with the following
formula for the expected signature of fractional Brownian motion with
Hurst parameter $H>\frac{1}{2}$ calculated in \cite{Fabric}.
\begin{prop}
(See \cite{Fabric}, Theorem 31 )\label{prop:main}Let $H>\frac{1}{2}$.
If $k=2n$ for some $n\in\mathbb{N}$, and $\left(i_{1},\ldots,i_{2n}\right)\in E_{2n}$,
then the projection to the basis $e_{i_{1}}\otimes e_{i_{2}}\ldots\otimes e_{i_{k}}$
of the expected signature of fractional Brownian Motion with Hurst
parameter $H$ up to time $T$ is

\[
\sum_{\pi\in\Pi_{i_{1},...,i_{2n}}}\left(H\left(2H-1\right)T^{H}\right)^{n}\int_{\triangle_{2n}\left(1\right)}\Pi_{\left(l,j\right)\in\pi}\left(u_{j}-u_{l}\right)^{2H-2}du_{1}...du_{2n}
\]
where $\triangle_{2n}\left(1\right)$ denotes the simplex $\left\{ \left(u_{1},\ldots,u_{2n}\right)\in\mathbb{R}^{2n}:0\leq u_{1}<\ldots<u_{2n}\leq1\right\} $.
The projection to the basis $e_{i_{1}}\otimes e_{i_{2}}\ldots\otimes e_{i_{k}}$
is zero otherwise.\end{prop}
\begin{proof}
(of Proposition \ref{prop:main})For $H>\frac{1}{2}$, let $K_{H}\left(t,s\right)$
be defined, for $t>s$, as 
\[
K_{H}\left(t,s\right)=c_{H}s^{\frac{1}{2}-H}\int_{s}^{t}\left|u-s\right|^{H-\frac{3}{2}}u^{H-\frac{1}{2}}\mathrm{d}u
\]
where $c_{H}=\left[\frac{H\left(2H-1\right)}{\beta\left(2-2H,H-\frac{1}{2}\right)}\right]^{\frac{1}{2}}$
and $\beta$ denotes the beta function. 

Then by \cite{NVV99}, the fractional Brownian motion $B_{\cdot}^{H}$
can be represented as 
\[
B_{t}^{H}=\int_{0}^{t}K_{H}\left(t,s\right)\mathrm{d}B_{s}
\]
where $\mathrm{d}B_{s}$ denotes integration in the sense of Itô.

Note that $K\left(s^{+},s\right)=0$ for all $s$, $K\left(\cdot,s\right)$
is differentiable and has a positive, integrable derivative, and thus
$K_{H}$ satisfies (K1) and (K2). 

Now note
\begin{eqnarray*}
K_{H}\left(t,s\right) & \leq & c_{H}s^{\frac{1}{2}-H}t^{H-\frac{1}{2}}\int_{s}^{t}\left|u-s\right|^{H-\frac{3}{2}}\mathrm{d}u\\
 & \le & \frac{c_{H}}{H-\frac{1}{2}}s^{\frac{1}{2}-H}t^{H-\frac{1}{2}}\left(t-s\right)^{H-\frac{1}{2}}\\
 & \leq & \frac{c_{H}}{H-\frac{1}{2}}s^{\frac{1}{2}-H}t^{2H-1}
\end{eqnarray*}

and thus 
\[
K_{H}\left(t,s\right)^{2}\leq\left[\frac{c_{H}}{H-\frac{1}{2}}\right]^{2}s^{1-2H}t^{4H-2}
\]

and the right hand side is integrable in $s$.

As $K\left(\cdot,s\right)$ is increasing and $K\left(s^{+},s\right)=0$,
$\left|K\right|\left((s,t],s\right)=K\left(t,s\right)-K\left(s^{+},s\right)=K\left(t,s\right)$
and (K3) is satisfied. 

Therefore, Theorem \ref{thm:gen} applies and Proposition \ref{prop:main}
follows from a change of variable. 
\end{proof}

\section{Right continuity at $H=\frac{1}{2}$}

We shall prove that the formula for the expected signature in Propositon
\ref{prop:main} reconciles with the expected signature of Brownian
motion when we take limit as $H\rightarrow\frac{1}{2}$. By the self-similarity
property of fractional Brownian motions, it is sufficient to establish
this continuity in the case $T=1$. First we recall the expected signature
of Brownian motion up to time $1$:
\begin{prop}
(\cite{Faw03},\cite{LV04})The $n^{th}$ level term of the expected
signature of Brownian motion up to time $1$ is 
\[
\frac{1}{2^{n}n!}\left(\sum_{i=1}^{d}e_{i}\otimes e_{i}\right)^{n}
\]

\end{prop}
Equivalently, the projection to any basis of the form $e_{i_{1}}\otimes e_{i_{1}}\otimes\cdots\otimes e_{i_{n}}\otimes e_{i_{n}}$
is equal to $\frac{1}{2^{n}n!}$. 

Note that a term $e_{i_{1}}\otimes e_{i_{2}}\cdots\otimes e_{i_{2n}}$
would appear in the expansion of $\left(\sum_{i=1}^{d}e_{i}\otimes e_{i}\right)^{n}$
if and only if the pairing $\pi_{n}:=\left\{ \left(1,2\right),\left(3,4\right),\ldots,\left(2n-1,2n\right)\right\} $
is in $\Pi_{i_{1},...,i_{2n}}$. Thus to prove the continuity of the
expected signature as $H\rightarrow\frac{1}{2}$, it suffices to prove 
\begin{enumerate}
\item Let $\pi_{n}$ denote the pairing $\left\{ \left(1,2\right),\left(3,4\right),\ldots,\left(2n-1,2n\right)\right\} $.
Then the integral 
\[
I_{n}:=\left(H\left(2H-1\right)\right)^{n}\int_{\triangle_{2n}}\Pi_{\left(l,j\right)\in\pi_{n}}\left(u_{j}-u_{l}\right)^{2H-2}du_{1}\ldots du_{2n}
\]
converges to $\frac{1}{2^{n}n!}$ as $H\rightarrow\frac{1}{2}$. 
\item The integral
\[
\left(H\left(2H-1\right)\right)^{n}\int_{\triangle_{2n}}\Pi_{\left(l,j\right)\in\pi}\left(u_{j}-u_{l}\right)^{2H-2}du_{1}.\ldots du_{2n},
\]
converges to $0$ as $H\rightarrow\frac{1}{2}$ for any $\pi\in E_{2n}\backslash\pi_{n}$.
\end{enumerate}
We will calculate first the integral $I_{n}$, which we may write
explicitly as 
\[
\left(H\left(2H-1\right)\right)^{n}\int_{\triangle_{2n}}\Pi_{j=1}^{n}\left(u_{2j}-u_{2j-1}\right)^{2H-2}du_{1}\ldots du_{2n}
\]

By Fubini's theorem, and using the notation $u_{2n+1}=1$, we can
``integrate $u_{2n}$s first'' to obtain: 
\[
\begin{array}{cc}
 & \left(H\left(2H-1\right)\right)^{n}\int_{\triangle_{2n}}\Pi_{j=1}^{n}\left(u_{2j}-u_{2j-1}\right)^{2H-2}du_{1}\ldots du_{2n}\\
= & \left(H\left(2H-1\right)\right)^{n}\int_{0}^{1}\int_{0}^{u_{2n}}...\int_{0}^{u_{2}}\Pi_{j=1}^{n}\left(u_{2j}-u_{2j-1}\right)^{2H-2}du_{1}\ldots du_{2n}\\
= & \left(H\left(2H-1\right)\right)^{n}\int_{0}^{1}\int_{0}^{u_{2n-3}}...\int_{0}^{u_{3}}\Pi_{j=1}^{n}\left[\int_{u{}_{2j-1}}^{u{}_{2j+1}}\left(u_{2j}-u_{2j-1}\right)^{2H-2}du_{2j}\right]du_{1}du_{3}\ldots du_{2n-1}\\
= & H^{n}\int_{0}^{1}\int_{0}^{u_{2n-3}}...\int_{0}^{u_{3}}\Pi_{j=1}^{n}\left(u_{2j+1}-u_{2j-1}\right)^{2H-1}du_{1}du_{3}du_{5}\ldots du_{2n-1}.
\end{array}
\]

Taking limit as $H\rightarrow\frac{1}{2}$ and using the bounded convergence
theorem, we have 
\[
\begin{array}{cc}
 & \lim_{H\rightarrow\frac{1}{2}}H^{n}\int_{0}^{1}\int_{0}^{u_{2n-3}}...\int_{0}^{u_{3}}\Pi_{j=1}^{n}\left(u_{2j+1}-u_{2j-1}\right)^{2H-1}du_{1}du_{3}du_{5}\ldots du_{2n-1}\\
= & \left(\frac{1}{2}\right)^{n}\int_{0}^{1}\int_{0}^{u_{2n-3}}...\int_{0}^{u_{3}}du_{1}du_{3}du_{5}\ldots du_{2n-1}\\
= & \frac{1}{2^{n}n!}.
\end{array}
\]

Now we consider other pairings $\pi$ in $\Pi_{i_{1},...,i_{2n}}$,
that is when there exists a $k$,$i$ such that $k-i>1$ but$\left(i,k\right)\in\pi$.
We have

\begin{equation}
\begin{array}{cc}
 & \left(H\left(2H-1\right)\right)^{n}\int_{\triangle_{2n}}\Pi_{\left(l,j\right)\in\pi}\left(u_{j}-u_{l}\right)^{2H-2}du_{1}\ldots du_{2n}\\
\leq & \left(H\left(2H-1\right)\right)^{n-1}\int_{\triangle_{2n-2}}\Pi_{\left(l,j\right)\in\pi\backslash\left(i,k\right)}\left[\left(u_{j}-u_{l}\right)^{2H-2}du_{l}du_{j}\right]\\
 & \times\left(H\left(2H-1\right)\right)\int\left(u_{k}-u_{i}\right)^{2H-2}1_{\left[u_{i-1},u_{i+1}\right]\text{\ensuremath{\times}}\left[u_{k-1},u_{k+1}\right]}\left(u_{i},u_{k}\right)du_{i}du_{k}.
\end{array}\label{eq:hh}
\end{equation}

Note that as $k-1\geq i+1$, 
\begin{eqnarray*}
 &  & H\left(2H-1\right)\int_{\mathbb{R}^{2}}\left(u_{k}-u_{i}\right)^{2H-2}1_{\left[u_{i-1},u_{i+1}\right]\text{\ensuremath{\times}}\left[u_{k-1},u_{k+1}\right]}\left(u_{i},u_{k}\right)du_{i}du_{k}\\
 & \leq & H\left(2H-1\right)\int_{\mathbb{R}^{2}}\left(u_{k}-u_{i}\right)^{2H-2}1_{\left[0,u_{k-1}\right]\text{\ensuremath{\times}}\left[u_{k-1},1\right]}\left(u_{i},u_{k}\right)du_{i}du_{k}\\
 & = & \frac{1}{2}\left[1-u_{k-1}^{2H}-\left(1-u_{k-1}\right)^{2H}\right]\\
 & \leq & \frac{1}{2}\left(\frac{1}{2}-\left(\frac{1}{2}\right)^{2H}\right),
\end{eqnarray*}
where the final inequality holds because $0\leq1-x^{2H}-\left(1-x\right)^{2H}\leq1-\left(\frac{1}{2}\right)^{2H-1}$
for $H>\frac{1}{2}$. 

Thus by (\ref{eq:hh}), 
\begin{eqnarray*}
 &  & \left(H\left(2H-1\right)\right)^{n}\int_{\triangle_{2n}}\Pi_{\left(l,j\right)\in\pi}\left(u_{j}-u_{l}\right)^{2H-2}du_{1}\ldots du_{2n}\\
 & \leq & \frac{1}{2}\left(1-\left(\frac{1}{2}\right)^{2H-1}\right)\left(H\left(2H-1\right)\right)^{n-1}\int_{\triangle_{2n-2}}\Pi_{\left(l,j\right)\in\pi\backslash\left(i,k\right)}\left[\left(u_{j}-u_{l}\right)^{2H-2}du_{l}du_{j}\right].
\end{eqnarray*}

Note first that 
\begin{eqnarray*}
 &  & \left(H\left(2H-1\right)\right)^{n-1}\int_{\triangle_{2n-2}}\Pi_{\left(l,j\right)\in\pi\backslash\left(i,k\right)}\left[\left(u_{j}-u_{l}\right)^{2H-2}du_{l}du_{j}\right]\\
 & \leq & \left(H\left(2H-1\right)\right)^{n}\int_{\left[0,1\right]^{2n}}\Pi_{\left(l,j\right)\in\pi}\left|u_{j}-u_{l}\right|^{2H-2}du_{1}\ldots du_{2n}\\
 & = & \left(2H\left(2H-1\right)\right)^{n}\left(\int_{0}^{1}\int_{0}^{1}\left|y-x\right|^{2H-2}dxdy\right)^{n}\\
 & = & \left(2H\left(2H-1\right)\right)^{n}\left(\int_{0}^{1}\int_{0}^{y}\left|y-x\right|^{2H-2}dxdy\right)^{n}\\
 & = & 1,
\end{eqnarray*}

therefore,
\begin{eqnarray*}
 &  & \left(H\left(2H-1\right)\right)^{n}\int_{\triangle_{2n}}\Pi_{\left(l,j\right)\in\pi}\left(u_{j}-u_{l}\right)^{2H-2}du_{1}\ldots du_{2n}\\
 & \leq & \frac{1}{2}\left(1-\left(\frac{1}{2}\right)^{2H-1}\right)\\
 & \rightarrow & 0
\end{eqnarray*}

as $H\rightarrow\frac{1}{2}$.

\end{document}